\documentclass[a4paper,11pt,reqno,dvips]{amsart}
\usepackage{amsmath,a4,amsfonts,amssymb,amsthm,enumerate,fleqn,geometry}
\usepackage{graphicx}
\usepackage[dvips]{epsfig}
\def\RR{\mathbb{R}}
\def\NN{\mathbb{N}}

\theoremstyle{plain}

\newtheorem{theorem}{\bf Theorem}[section]

\theoremstyle{remark}
\newtheorem{definition}[theorem]{\bf Definition}
\newtheorem{observation}[theorem]{\bf Observation}
\newtheorem{example}[theorem]{\bf Example}
\newtheorem{remark}[theorem]{\bf Remark}

\newtheorem{corollary}[theorem]{\bf Corollary}

\title[Computability of Frames In Computable Hilbert Spaces]{Computability of Frames In Computable Hilbert Spaces}

\author[Poonam Mantry]{Poonam Mantry}
\address{Daulat Ram College}
\email{poonam.mantry@gmail.com}

\author[S.K. Kaushik]{S.K. Kaushik}
\address{Kirorimal College}
\email{shikk2003@yahoo.co.in}

\begin{document}

\begin{abstract}
Frames play an important role in various practical problems related to signal and image processing. In this paper, we define computable frames in computable Hilbert spaces and obtain computable versions of some of their characterizations. Also, the notion of duality of frames in the context of computability has been studied. Finally, a necessary and sufficient condition for the existence of a computable dual frame is obtained.
\end{abstract}

\thispagestyle{empty}
\maketitle

\thispagestyle{empty}

\section{Introduction}
Frames are generalizations of orthonormal basis in Hilbert spaces. The notion of frames was introduced in 1952 by Duffin and Schaeffer \cite{7}. Basis in a Hilbert space $H$ allows every $f\in H$ to be represented as a unique expansion in terms of the basis elements. However, the condition of uniqueness of expansion is very restrictive in nature. Frames are particularly important because they provide the desired flexibility. The redundant nature of frames is a desirable property in many practical problems.\\
Frames are characterized by the associated synthesis and analysis operator. Frames allow every element in a Hilbert space to have a representation as an infinite linear combination of the frame elements. This representation is known as frame decomposition.\\
In this paper, we extend the notion of computability to frames in Hilbert spaces. We prove the computable versions of the characterizations of frames in the framework of computable analysis. Similar to Fourier representation of a computable Hilbert space $H$ given by Brattka \cite{3}, we define Frame representation for $H$ and give conditions under which it coincides with the Fourier representation of $H$. Finally, the notion of computable dual frame is defined and a necessary and sufficient condition for its existence is proved.

\section{Preliminaries}
In this section, we briefly summarize some notions from computable analysis as presented in \cite{15}. Computable Analysis is the Turing machine based approach to computability in analysis. Pioneering work in this field has been done by Turing \cite{14}, Grzegorczyk \cite{8}, Lacombe\cite{11}, Banach and Mazur\cite{1}, Pour-El and Richards\cite{13}, Kreitz and Weihrauch\cite{10} and many others. The basic idea of the representation based approach to computable analysis is to represent infinite objects like real numbers, functions or sets by infinite strings over some alphabet $\Sigma$ (which at least contains the symbols 0 and 1). Thus, a representation of a set X is a surjective function $\delta:\subseteq \Sigma^{\omega}\rightarrow X$ where $\Sigma^{\omega}$ denotes the set of infinite sequences over $\Sigma$ and the inclusion symbol indicates that the mapping might be partial. Here, $(X,\delta)$ is called a represented space. \\
A function $F:\subseteq\Sigma^\omega \to \Sigma^\omega$ is said to be computable if there exists some Turing machine, which computes infinitely long and transforms each sequence $p$, written on the input tape, into the corresponding sequence $F(p)$, written on one way output tape.  Between two represented spaces, we define the notion of a computable function.
\begin{definition}\cite{5}
Let $(X,\delta)$, $(Y,\delta^\prime)$ be represented spaces. A function $f:\subseteq X\to Y$ is called $(\delta,\delta^\prime)$-computable if there exists a computable function $F:\subseteq\Sigma^\omega \to \Sigma^\omega$ such that $\delta^\prime F(p)= f\delta(p)$, for all $p\in dom(f\delta)$.
\end{definition}
We simply call, a function $f$ computable, if the represented spaces are clear from the context. For comparing two representations $\delta$, $\delta^{'}$of a set $X$, we have the notion of reducibility of representations. $\delta$ is called reducible to $\delta^{'}$, $\delta \leq \delta^{'}$ (in symbols), if there exists a computable function $F:\subseteq \Sigma^{\omega}\rightarrow \Sigma^{\omega}$ such that $\delta(p)= \delta^{'}F(p)$, for all $p\in dom(\delta)$. This is equivalent to the fact that the identity $I:X\rightarrow X$ is $(\delta, \delta^{'})$- computable. If $\delta \leq \delta^{'}$ and $\delta^{'} \leq \delta$, then $\delta$ and $\delta^{'}$ are called computably equivalent.
Analogous to the notion of computability, we can define the notion of $(\delta,\delta^{'})$-continuity, by substituting a continuous function $F$ in the Definition 2.1. On $\Sigma^{\omega}$, we use the Cantor topology, which is simply the product topology of the discrete topology on $\Sigma$.
\\Given a represented space $(X,\delta)$, a computable sequence is defined as a computable function $f:\NN\to X $ where we assume that $\NN $ is represented by $\delta_{\NN}(1^{n}0^{\omega})= n$ and a point $x\in X$ is called computable if there is a constant computable sequence with value $x$. The notion of $(\delta,\delta^{'})$-continuity agrees with the ordinary topological notion of continuity, as long as, we are dealing with admissible representations.
\\ A representation $\delta$ of a topological space $X$ is called admissible if $\delta$ is continuous  and if the identity $I:X\to X$ is $(\delta^\prime,\delta)$-continuous for any continuous representation $\delta^\prime$ of $X$. If $\delta$, $\delta^\prime$, are admissible representation of topological spaces $X$,$Y$, then a function $f:\subseteq X\to Y$ is $(\delta,\delta^\prime)$ continuous if and only if it is sequentially continuous \cite{2}.
\\Given two represented spaces $(X,\delta)$, $(Y,\delta^\prime)$, there is a canonical representation $[\delta\rightarrow\delta^\prime]$ of the set of $(\delta,\delta^\prime)$-continuous functions $f:X\to Y $. If $\delta$ and $\delta^\prime$ are admissible representations of sequential topological spaces $X$ and $Y$ respectively, then $[\delta\rightarrow\delta^\prime]$ is actually a representation of the set $C(X,Y)$ of continuous functions $f:X\rightarrow Y$. The function space representation can be characterized by the fact that it admits evaluation and type conversion. See \cite{15} for details.
\\ If $(X,\delta)$, $(Y,\delta^\prime)$ are admissibly represented sequential topological spaces, then, in the following, we will always assume that $C(X,Y)$  is represented by $[\delta\rightarrow\delta^\prime]$ . It follows by evaluation and type conversion that the computable points in $(C(X,Y),[\delta\rightarrow\delta^\prime])$ are just the $(\delta,\delta^\prime)$-computable functions $f:\subseteq X\to Y$ \cite{15}. For a represented space $(X,\delta)$, we assume that the set of sequences $X^{\NN}$ is represented by $\delta^{\NN} \equiv [\delta_{\NN}\rightarrow \delta]$. The computable points in $(X^{\NN},\delta^{\NN})$ are just the computable sequences in $(X,\delta)$.
\\The notion of computable metric space was introduced by Lacombe \cite{12}. However, we state the following definition given by Brattka \cite{4}.
\begin{definition}\cite{4}
 A tuple $(X,d,\alpha)$ is called a \emph{computable metric space} if
\begin{enumerate}
   \item $(X,d)$ is a metric space.
   \item $\alpha:\NN\to X$ is a sequence which is a dense in $X$.
   \item $do(\alpha\times\alpha):\NN^2\to \RR$ is a computable (double) sequence in $\RR$.

\end{enumerate}
\end{definition}
Given a computable metric space $(X,d,\alpha)$, its Cauchy Representation $\delta_{X}:\subseteq \Sigma^{\omega}\rightarrow X$ is defined as\\
 \begin{align*} \delta_{X}(01^{n_{0}+1}01^{n_{1}+1}01^{n_{2}+1}...):=  \lim_{i\rightarrow\infty}\alpha(n_{i}), \end{align*} \\for all $n_{i} \in \NN$ such that $(\alpha(n_{i}))_{i\in \NN}$ converges  and \\
\begin{align*} d(\alpha(n_{i}),\alpha(n_{j})) < 2^{-i}, \ \text{for all} \ j>i.\end{align*} \\
\\In the following, we assume that computable metric spaces are represented by their Cauchy representation. All Cauchy representations are admissible with respect to the corresponding metric topology.
\\An example of a computable metric space is $(\RR,d_{\RR},\alpha_{\RR})$ with the Euclidean metric given by $d_{\RR}(x,y)= \|x-y\|$ and a standard numbering of a dense subset $Q\subseteq \RR$ as $\alpha_{\RR}\langle i,j,k \rangle = (i-j)/(k+1)$. Here, the bijective Cantor pairing function $\langle \cdot , \cdot \rangle: \NN^{2}\rightarrow \NN$ is defined as $\langle i,j \rangle = j+(i+j)(i+j+1)/2$ and this definition can be extended inductively to finite tuples. It is known that the Cantor pairing function and the projections of its inverse are computable. In the following, we assume that $\RR$ is endowed with the Cauchy representation $\delta_{\RR}$ induced by the computable metric space given above.
\\Brattka gave the following definition of a computable normed linear space.
\begin{definition} \cite{4}
A space $(X,\|\cdot\|,e)$ is called a \emph{computable normed space} if
 \begin{enumerate}
 \item $\|\cdot\|: X\rightarrow \RR$ is a norm on $X$.
 \item The linear span of $e:\NN\rightarrow X$ is dense in $X$.
 \item $(X,d,\alpha_{e})$ with $d(x,y)= \parallel x-y\parallel$ and $\alpha_{e}\langle k,\langle n_{0},...,n_{k}\rangle \rangle= \sum_{i=0}^{k}\alpha_{\mathbb{F}}(n_{i})e_{i}$ is a computable metric space with Cauchy representation $\delta_{X}$.

\end{enumerate}
Here, $\alpha_\mathbb{F}$ is a standard numbering of $\mathbb{Q}_\mathbb{F}$ where $\mathbb{Q}_\mathbb{F}$= $\mathbb{Q}$ in case of $\mathbb{F} = \mathbb{R}$ and $\mathbb{Q}_\mathbb{F}= \mathbb{Q}[i]$ in case of $\mathbb{F}= \mathbb{C}$. 
\end{definition}
 It was observed that computable normed space is automatically a computable vector space, i.e., the linear operations are all computable. If the underlying space $(X,\|\cdot\|)$ is a Banach space then $(X,\|\cdot\|,e)$ is called a \emph{computable Banach space}. If $X$ is a computable Banach space, then a sequence $(e_i)_{i\in \mathbb{N}}$ is said to be a \emph{computable basis} if it is a Schauder Basis that is computable in $X$. 
 \\We always assume that computable normed spaces are represented by their Cauchy representations, which are admissible with respect to norm topology. Two computable Banach spaces with the same underlying set are called computably equivalent if the corresponding Cauchy representations are computably equivalent.\\
In this paper, we will discuss operators on computable Hilbert spaces. Brattka gave the following definition of computable Hilbert spaces.
\begin{definition}\cite{3}
A \emph{computable Hilbert space} $(H,\langle \cdot \rangle,e)$ is a separable Hilbert space $(H,\langle.\rangle)$ together with a fundamental sequence $e:\NN \rightarrow H$ such that the induced normed space is a computable normed space.
\end{definition}
If $H$ be a computable Hilbert space with some fixed computable orthonormal basis $(e_n)$, then the Fourier representation $\delta_{Fourier}$ of $H$ is given by
\begin{align} \delta_{Fourier}\langle p,q \rangle = x \Leftrightarrow \delta_{\mathbb{F}}^\mathbb{N}(p)= (\langle x, e_n \rangle)_{n\in \mathbb{N}} \ \text{and} \ \delta_\mathbb{R}(q)= \lVert x\rVert \end{align} 
for all $x\in H $. The Fourier representation of $H$ with respect to this basis is computably equivalent to the Cauchy representation of $H$. The standard dual space representation of the dual space $H^{'}$ is given by 
\begin{align} \delta_{H^{'}}\langle p,q \rangle = f \Leftrightarrow [\delta_H \rightarrow \delta_\mathbb{F}](p)=f\ \text{and} \ \delta_\mathbb{R}(q)= \lVert f\rVert\end{align}
for all $f\in H^{'}$. A computable isomorphism $T: H\rightarrow H$ is an isomorphism such that $T$ as well as $T^{-1}$ are computable.

A sequence $(f_{i})_{i\in \NN}$ of elements in a Hilbert space $H$ is called a \emph{frame} for $H$ if there exist constants $A, B > 0$ such that
\begin{align*} A\|f\|^{2} \leq \sum_{i\in \NN}|\langle f,f_{i}\rangle|^{2} \leq B\|f\|^{2}, \ \text{for all} \ f\in H. \end{align*}
If only the upper inequality holds in above, then $(f_{i})_{i\in \NN}$ is said to be a \emph{Bessel sequence} in $H$.
The numbers $A$ and $B$ are called a lower and upper frame bound for the frame $(f_{i})_{i\in \NN}$. The \emph{synthesis operator} $T$ given by $T((c_{i})) = \sum_{i=1}^{\infty} c_{i}f_{i}, (c_{i})_{i\in\mathbb{N}}\in l^{2}$ associated with a frame $(f_{i})$ is a bounded operator from $l^{2}$ onto $H$. The adjoint operator $T^{*} : H \rightarrow l^{2}$ given by $T^{*}(f)= (\langle f,f_{i}\rangle)_{i\in \NN}$ is called the \emph{analysis operator}. The \emph{frame operator} $S : H\rightarrow H$ given by $S(f) = TT^{*}(f)= \sum_{i\in \NN}\langle f,f_{i} \rangle f_{i}$ associated with the frame $(f_{i})$ is a bounded, invertible and positive operator mapping $H$ onto itself. The \emph{frame decomposition}, stated below, shows that every element in $H$ has a representation as a superposition of the frame elements.
\begin{align}
f=\sum_{i\in \NN}\langle f, S^{-1}f_i\rangle f_i= \sum_{i\in \NN}\langle f,f_i\rangle S^{-1}f_i
\end{align}
for all $f\in H$. Both the series converge unconditionally for all $f\in H$. Thus, it is natural to view a frame as some kind of generalized basis. A frame is said to be \emph{exact} if it ceases to be a frame when an arbitrary element is removed. A nonexact frame is said to be \emph{overcomplete}. Frames for $H$ are characterized as the families $(Ue_{k})_{k\in \mathbb{N}}$ where $(e_k)_{k\in \mathbb{N}}$ is an orthonormal basis for $H$ and $U: H\rightarrow H$ is a bounded surjective operator.
\par A sequence $(x_n)$ in a Hilbert space $H$ is said to be a \emph{Riesz Basis} if it is equivalent to some orthonormal basis for $H$. That is, if there exists an orthonormal basis $(e_n)$ for $H$ and a topological isomorphism $T: H\rightarrow H$ such that $x_n= Te_n$ for every $n\in \mathbb{N}$.
 For other concepts and results related to frames, refer to \cite{6}.

\section{Main Results}

\setcounter{equation}{0}

We begin with the following definitions of a computable frame and computable Reisz basis in a computable Hilbert space $H$.
\begin{definition}
Let $H$ be a computable Hilbert space. A \emph{computable frame} for $H$ is a computable sequence $(f_{i})\subseteq H $ for which there are constants $A$,$ B > 0 $ satisfying
\begin{align*} A\|f\|^{2} \leq \sum_{i\in \NN}|\langle f,f_{i}\rangle|^{2} \leq B\|f\|^{2} \end{align*}
for all $f\in H$.
\end{definition}
\begin{remark}Every computable Hilbert space has a computable frame. \end{remark}
\begin{definition}
\emph{Computable Riesz basis} in a computable Hilbert space is a Riesz basis which is also computable as a sequence. That is, if $(H,\|.\|,(e_{n}))$ be a computable Hilbert space with orthonormal basis $(e_{n})$, then the sequence $(x_{n}) = (Te_{n})$ is said to be a computable Riesz basis for H if and only if $T:H\rightarrow H$ is a computable isomorphism.
\end{definition}
A Reisz basis $(x_{n})$ in a Hilbert space $H$ induces another equivalent norm on $H$ defined as $\||x\|| =\|Tx\|, x \in H$. The following result shows that the two norms are computably equivalent as well.

\begin{theorem}
 If $(H,\|.\|,(e_{n}))$ be a computable Hilbert space with orthonormal basis $(e_{n})$ and $(x_{n})$ = $(Te_{n})$ be a computable Riesz basis for H. Then $(H,\||.\||,(x_{n}))$ is a computably equivalent Hilbert space where $\||x\|| =\|Tx\|, x \in H$.
\end{theorem}
\begin{proof}
 The norm $\||.\||$ is equivalent to the norm $\|.\|$ of $H$ by Theorem 7.13 \cite{9}. Let $\delta_{H}^{'}$ be the Cauchy representation of the Hilbert space
$(H,\||.\||,(e_{n}))$. The metric induced by the norm $\||.\||$ gives $d(\alpha_{e}(n),\alpha_{e}(m))= \|T(\alpha_{e}(n)-\alpha_{e}(m))\|$ which is $\delta_{\RR}$ computable as T is computable. Altogether, $(H,\||.\||,(e_{n}))$ forms a computable Hilbert space. Now, given $\delta_{H}$ name of $x\in H$ and precision $k \in \NN$, we can effectively find $n\in \mathbb{N}$ and coefficients $a_{0}, a_{1},...a_{n} \in \mathbb{F}$ such that $\| x-\sum_{i=0}^{n} a_{i}e_{i}\| < 2^{-k}/\|T\|$. Since \begin{align*}\||x-\sum_{i=0}^{n}a_{i}e_{i}\|| \leq \|T\| \|x-\sum_{i=0}^{n}a_{i}e_{i}\| \leq 2^{-k} \end{align*}
 the $\delta_{H}^{'}$ name of $x$ can be computed and so the identity map $I:(H,\|.\|,(e_{n}))\rightarrow (H,\||.\||,(e_{n})) $ is computable and by computable Banach's Inverse Mapping Theorem \cite{4}, $I^{-1}$ is computable as well. Since $(x_{n})$ is a computable orthonormal basis for $H$ with respect to $\||.\||$, $(H,\||.\||,(x_{n}))$ is a computably equivalent Hilbert space.
\end{proof}

The following result gives a computable version of a characterization of frames in terms of the synthesis operator.

\begin{theorem}Let $H$ be a computable Hilbert space. A sequence $(f_{i})\subseteq H $ is a  computable frame for $H$ if and only if the synthesis operator
$T : l^{2} (\NN) \rightarrow H$ given by $(c_{k}) \mapsto \sum_{k=1}^{\infty} c_{k}f_{k}$ is a computable operator from $l^{2}$ onto $H$.
\end{theorem}
\begin{proof}
 As $(f_{i})$ is a frame for $H$, $T$ is a well defined bounded linear operator from $l^{2}$ onto $H$. Let $(\delta_{k})$ be the standard computable orthonormal basis of $l^{2}$. Then, $(T(\delta_{k})) = (f_{k})$ is a computable sequence in $H$. Therefore, $T$ is a computable operator from $l^{2}$ onto $H$. \\
Conversely, since $T$ is a well defined mapping of $l^{2}$ onto $H$, $(f_{i})$ is a frame for $H$. As a computable operator maps computable sequences to computable sequences, $ (f_{k})= (T(\delta_{k}))$ is a computable sequence in $H$. Thus, $(f_{i})$ is a computable frame for $H$.
\end{proof}

 Next, we obtain computable version of another characterization of frames in the following result.

\begin{theorem} Let  $(H,\|\cdot\|,(e_{i}))$ be a computable Hilbert space, $(e_{i})$ being the computable orthonormal basis. Computable frames for $H$ are the families $(U(e_{k}))$, where $U: H\rightarrow H$ is a computable surjective operator.
\end{theorem}
\begin{proof}
 Let $(f_{k})$ be a computable frame for $H$. Consider the isometrical isomorphism
$\phi : H \rightarrow l^{2}$ given by $\phi(e_{k})= \delta_{k}$, where $(\delta_{k})$ is the standard computable orthonormal basis of $l^{2}$. Clearly, $\phi$ is a computable map. Write $U= T\phi$, $T$ being the synthesis operator. Then, $U$ is a computable surjective operator such that $U(e_{k})= T(\delta_{k})= f_{k}$.
\\Conversely, let $U$ be any computable surjective operator. Then, $(U(e_{k}))$ is computable sequence in $H$ such that $\sum_{k\in \NN}|\langle f,Ue_{k}\rangle|^{2} \leq \|U\|^{2} \|f\|^{2}$, for all $f\in H$. Since $U$ is surjective, there exists a constant $C > 0$ such that $ C^{2} \|f\|^{2} \leq \sum_{k\in \NN}|\langle f,Ue_{k}\rangle|^{2}$, for all $f\in H$. Thus, $(U(e_{k}))$ is a computable frame for $H$.\end{proof} 

Now, we consider the analysis operator associated to a frame $(f_{i})$ in a Hilbert space $H$. Let $T: l^{2}\rightarrow H$ be the synthesis operator. Consider the adjoint operator of $T$ namely, $T^{*} : H \rightarrow l^{2}$ given by $T^{*}(f)= (\langle f,f_{i}\rangle)$, for all $f\in H$. It is proved in \cite{9} that a sequence $(f_{i})$ is a frame for $H$ if and only if the analysis operator $T^{*}$ maps $H$ bijectively onto a closed subspace of $l^{2}$. In the following example, we show that for a computable frame, the associated analysis operator need not be computable.

\begin{example}
Let $(a_{i})$ be a computable sequence of positive reals such that $\|(a_{i})\|_{l^{2}}$ exists but is not computable. We assume $a_{0}= 1$ and $\|(a_{i})\|_{l^{2}}^2 < 2$. Using this sequence, define a linear bounded operator $U: l^{2}\rightarrow l^{2}$ as\\

\[\left(\begin{array}{ccccc}
 1& a_{1}& a_{2} & a_{3}& \cdots\\
 0& 1 & 0 & 0 & \cdots\\
 0 & 0 & 1 & 0 & \cdots\\
 \vdots &\vdots &\vdots & \ddots\end{array}\right).\]

 $U$ is a computable operator and since $U$ is surjective, $(f_{i})= (U(\delta_{i}))$ is a computable frame for $l^{2}$. However, the analysis operator $T^{*}$ is not computable as $T^{*}(\delta_{0})= (\langle \delta_{0},U\delta_{i} \rangle )=(a_{i})$, which is not computable in $l^{2}$.

\end{example}

\begin{observation}
One may observe that in the above example, $U^{*}$ given by\\

    \[\left(\begin{array}{ccccc}
    1& 0 & 0 &0& \cdots\\
    a_{1}& 1 & 0 & 0 & \cdots\\
    a_{2} & 0 & 1 & 0 & \cdots\\
    \vdots &\vdots &\vdots & \ddots\end{array}\right)\]

is not computable.\end{observation}

In the following result, we give a sufficient condition for the computability of the analysis operator.

\begin{theorem}
Let $(H,\|\cdot\|,(e_{i}))$ be a computable Hilbert space, where $(e_{i})$ is the computable orthonormal basis. Let $(f_{n})= (U(e_{n}))$ be a computable frame for $H$, where $U$ is a computable surjective operator. If $U^{*}$, the adjoint operator of $U$, is computable, then the analysis operator $T^{*} : H \rightarrow l^{2}$ given by $T^{*}(f)= (\langle f,f_{i} \rangle )_{i\in \NN}$, for all $f\in H$ is $(\delta_{H},\delta_{l^{2}})$ computable.
\end{theorem}
\begin{proof}
	 Let $f\in H$ be represented with respect to the Cauchy representation $\delta_{H}$. Since the inner product on $H$ is computable and $(f_{i})$ is a computable sequence in $H$, $(\langle f,f_{i} \rangle )_{i\in \NN}$ is a computable sequence in $\mathbb{F}$. Also, we have
\begin{align*}\sum_{i\in \NN}|\langle f,f_{i}\rangle|^{2} = \sum_{i\in \NN}|\langle f,Ue_{i}\rangle|^{2} = \|U^{*}f\|^{2} , f\in H.\end{align*}
Now, using hypothesis, the evaluation property and the computability of $\|\cdot\|$, $\|U^{*}f\|^{2}$ is $\delta_{\RR}$ computable. Therefore, a Cauchy name of the sequence $(\langle f,f_{i} \rangle ) \in l^{2}$ can be computed. Hence, $T^{*}$ is computable.
\end{proof}

\begin{corollary}
Let $(H,\|\cdot\|,(e_{i}))$ be a computable Hilbert space, $(e_{i})$ being the computable orthonormal basis. Let $(f_{n})= (U(e_{n}))$ be a computable frame for $H$, where $U$ is a computable surjective operator. If $U^{*}$ is computable, then the frame operator $S = TT^{*}: H \rightarrow H$ given by $Sf = \sum_{i\in \NN}\langle f,f_{i} \rangle f_{i}$, $f\in H$ is a computable isomorphism.
\end{corollary}
\begin{proof}
 Since $(f_{k})$ is a frame, $S$ is an isomorphism and computability of $S$ follows from the computability of $T$ and $T^{*}$. By Computable Banach Inverse Mapping Theorem \cite{4}, the operator $S^{-1}$ is also computable.
\end{proof}

\begin{corollary}
Let $(H,\|\cdot\|,(e_{i}))$ be a computable Hilbert space, $(e_{i})$ being the computable orthonormal basis. Let $(f_{n})= (U(e_{n}))$ be a computable frame for $H$, where $U$ is a computable surjective operator. If $U^{*}$ is computable, then the operator $T^{+} : H \rightarrow l^{2}$ given by $T^{+}(f)= (\langle f,S^{-1}f_{k} \rangle)_{k \in \NN}$, for all $f\in H$ is a computable operator.
\end{corollary}
\begin{proof}
 Let $f\in H$ be represented with respect to the Cauchy representation. Then by Corollary 3.10, $\delta_H$ name of $S^{-1}f$ can be computed. As $T^{*}$ is a computable operator, $\delta_{l^{2}}$ name of $T^{*}(S^{-1}f)= (\langle S^{-1}f,f_{k}\rangle )= (\langle f,S^{-1}f_{k}\rangle)$ can also be computed.
\end{proof}

The operator $T^{+}$ is called the \emph{pseudo-inverse} of the synthesis operator $T$. The numbers $\langle f,S^{-1}f_{k}\rangle $ are called the \emph{frame coefficients}. In \cite{3}, Brattka defined the Fourier representation $\delta_{Fourier}$ 
of a computable Hilbert space $H$ using the Fourier coefficients corresponding to a fixed orthonormal basis. This motivates the following representation of $H$ using the frame coefficients.

\begin{definition}
Let $H$ be a computable Hilbert space with a fixed computable frame $(f_{i})$. Define the frame representation of $H$ as follows:
\begin{align*} \delta_{Frame}(\langle p,q \rangle)= f  \Leftrightarrow  \delta_{\mathbb{F}}^{\NN}(p)= (\langle f, S^{-1}f_{k} \rangle)_{k\in \mathbb{N}} \ \ and \ \ \delta_{\RR}(q)= \sum_{k=1}^{\infty}|\langle f,S^{-1}f_{k} \rangle|^{2}, \end{align*}
\end{definition}

where $\delta_{\mathbb{F}}^{\NN}$ denotes the canonical representation induced by $\delta_{\mathbb{F}}$. That is, one can read the frame representation such that any point $f\in H$ is considered as $(\langle f, S^{-1}f_{k} \rangle ) \in l^{2}$.\\

We now analyze the relationship between the above defined frame representation $\delta_{Frame}$ and the Cauchy representation   $\delta_H$ of a computable Hilbert Space $H$. We observe that $\delta_{Frame} \leq \delta_{H}$. However, $\delta_{H}\leq \delta_{Frame}$ under some additional assumptions.

\begin{theorem}
Let $(H,\|\cdot\|,(e_{i}))$ be a computable Hilbert space with a computable frame $(f_{i})= (Ue_{i})$, $(e_{i})$ being the computable orthonormal basis. The Frame representation of $H$ with respect to $(f_{i})$ is reducible to the Cauchy representation of $H$. The converse statement holds, if $U^{*}$ is computable.
\end{theorem}
\begin{proof}
 Let $f\in H$ and $p,q \in \Sigma^{\omega}$ be such that $\delta_{Frame}(\langle p,q \rangle)= f$. Then,
\begin{align*} \delta_{\mathbb{F}}^{\NN}(p)= (\langle f, S^{-1}f_{k}\rangle ) \ \ and \ \ \delta_{\RR}(q)= \sum_{k=1}^{\infty}|\langle f,S^{-1}f_{k}\rangle|^{2}. \end{align*}
This gives $\delta_{l^{2}}(\langle p,q \rangle)= (\langle f, S^{-1}f_{k}\rangle)$. Using Theorem 3.5, $\delta_{H}$ name of $\sum_{k=1}^{\infty} \langle f,S^{-1}f_{k}\rangle f_{k}$ can be computed. Since $(f_{i})$ is a frame for $H$, by Frame decomposition, $f= \sum_{k=1}^{\infty}\langle f,S^{-1}f_{k} \rangle f_{k}$ for all $f\in H$. Hence, $\delta_{H}$ name of given $f$ can be computed.\\
For the converse part, given $\delta_{H}$ name of $f\in H$, we can compute $\delta_{l^{2}}$ name of $T^{+}f$ by Corollary 3.11. That is, we can get $\delta_{\mathbb{F}}^{\NN}$ name of $(\langle f, S^{-1}f_{k} \rangle)$ and $\delta_{\RR}$ name of $\sum_{k=1}^{\infty}|\langle f,S^{-1}f_{k} \rangle|^{2}$ and thus the $\delta_{Frame}$ name of $f$.\end{proof} 

In view of Example 3.7, given a computable frame, the analysis operator associated to the frame need not be computable and so in Theorem 3.9, we obtain a sufficient condition under which it becomes computable. We now consider the converse question, that is, given a frame, if the associated analysis operator is computable, then is the frame computable?
The following is a counterexample to this.

\begin{example}
Let $(a_{i})$ be a computable sequence of positive real numbers such that $\|(a_{i})\|_{l_{2}}$ exists but is not computable. We assume $a_{0}= 1$ and $\|(a_{i})\|_{l^{2}}^2 < 2$. Define a linear bounded operator $U: l^{2}\rightarrow l^{2}$ as\\

\[\left(\begin{array}{ccccc}
 1& 0 & 0 & 0 & \cdots\\
 a_{1}& 1 & 0 & 0 & \cdots\\
 a_{2} & 0 & 1 & 0 & \cdots\\
 \vdots &\vdots &\vdots & \ddots\end{array}\right).\]

 Since $U$ is a surjective operator, $(U\delta_{n})= (f_{n})$ forms a frame for $l^{2}$. The operator $U$ is not computable as $U\delta_{0}= (a_{i})$ is not computable in $l^{2}$. Therefore, $(f_{i})$ is a frame but not a computable frame. However, the analysis operator $T^{*} : H \rightarrow l^{2}$ given by $T^{*}(f)= (\langle f,f_{i} \rangle)$, $f\in H$ is the computable operator $U^{*}$. 
\end{example}

Note that in Example 3.14, $\|f_{0}\|$ is not computable. The following result gives a sufficient condition under which computability of analysis operator implies the computability of the frame.

\begin{theorem}
Let $(H,\|\cdot\|,(e_{i}))$ be a computable Hilbert space with $(e_{i})$ as the computable orthonormal basis. If the analysis operator for a frame $(f_{i})$, with computable sequence of norms $(\|f_{i}\|)$, is computable, then the frame is computable.
\end{theorem}
\begin{proof}
 Since the map $T^{*} : H \rightarrow l^{2}$ given by $T^{*}(f)= (\langle f,f_{i} \rangle)$, $f\in H$ is $(\delta_{H},\delta_{l^{2}})$ computable, $(T^{*}(e_{n}))$ is a computable sequence in $l^{2}$. So, there exists a computable map $g : \NN \times \NN \rightarrow \mathbb{F}$ given by $g(n,i) = \langle e_{n},f_{i}\rangle$. By Type conversion, the map $h : \NN  \rightarrow \mathbb{F}^{\NN}$ given by $h(i) = (\langle e_{n},f_{i}\rangle )_{n}$, $i\in \NN$ is $[\delta_{\NN},[\delta_{\NN} \rightarrow \delta_{\mathbb{F}}]]$ computable. Since $(f_{i})$ is a frame with a computable sequence of norms, given $i\in \NN$, we get the $\delta_{Fourier}$ name of $f_{i}$ and hence, $\delta_{H}$ name of $f_{i}$. Thus, $(f_{i})$ is a computable frame.\end{proof}

Corollary 3.11 shows that given a computable frame $(f_{i}) = (U(e_{i}))$ with $U^{*}$ computable, $((\langle f_n, S^{-1}f_{k}\rangle)_k)_n$ is a computable sequence in $l^{2}$. Conversely, we have the following result.

\begin{theorem}
Let $H$ be a computable Hilbert space and $(f_{i})$ be a frame for $H$. If $((\langle f_{n}, S^{-1}f_{k}\rangle)_{k})_n$ is a computable sequence in $\mathbb{F}^{\NN}$, then the frame $(f_{i})$ is a computable frame in $H$.
\end{theorem}
\begin{proof}
 By hypothesis, the map $G: \NN \rightarrow F^{\NN}$ given by $n \mapsto (\langle f_{n}, S^{-1}f_{k}\rangle)_{k}$, $n\in \NN$ is computable. By Theorem 8.22 in \cite{9}, we know that
   \begin{align*} \sum_{k=1,k\neq n}^{\infty}| \langle f_{k}, S^{-1}f_{n} \rangle|^{2} = \frac{1- |\langle S^{-1}f_{n}, f_{n} \rangle|^{2}- |1-\langle f_{n},S^{-1}f_{n}\rangle |^{2}}{2}. \end{align*}\\
This gives
  \begin{align*} \sum_{k=1}^{\infty}|\langle f_{n}, S^{-1}f_{k}\rangle |^{2} = \frac{1- |\langle S^{-1}f_{n}, f_{n}\rangle|^{2}- |1-\langle f_{n},S^{-1}f_{n}\rangle|^{2}}{2} + |\langle f_{n}, S^{-1}f_{n}\rangle |^{2}. \end{align*}
Now, given $\delta_\mathbb{N}$ name of $n\in \NN$, we can compute $\langle f_{n}, S^{-1}f_{n}\rangle $ and hence, $\delta_{\RR}$ name of $\sum_{k=1}^{\infty}|\langle f_{n}, S^{-1}f_{k}\rangle |^{2}$ can be computed. Thus, the map $G: \NN \rightarrow l^{2}$ given by $n \mapsto (\langle f_{n}, S^{-1}f_{k}\rangle)_{k} $, $n\in \NN$ is computable. Therefore, given $\delta_\mathbb{N}$ name of $n\in \NN$, we get $\delta_{Frame}$ name of $f_{n}$ and thus,  $\delta_{H}$ name of $f_{n}$. Hence $(f_{n})$ is a computable frame.\end{proof} 

\par Given a fixed frame $(f_i)$ in $H$, via frame decomposition, any other sequence $(\phi_l)$ in $H$ can be expressed as $\phi_l=\sum_{k=1}^{\infty} u_{l k} f_k$, where $u_{l k}=\langle\phi_l, S^{-1} f_k\rangle$. If $(\phi_l)$ is a frame, the operator defined by $\{u_{l k}\}_{l,k\in \mathbb{N}}$, that is,  $U_{\phi}:l^{2}\rightarrow l^{2}$ given by
\[\left(\begin{array}{ccccc}
\langle\phi_{1},S^{-1}f_{1}\rangle & \langle\phi_{1},S^{-1}f_{2}\rangle & \cdots\\
\langle\phi_{2},S^{-1}f_{1}\rangle & \langle\phi_{2},S^{-1}f_{2}\rangle & \cdots\\
\vdots &\vdots &\vdots & \end{array}\right)\]
 is a bounded linear operator on $l^{2}(\NN)$ by Proposition 5.5.6 [7].\\
 Conversely, recall that if $(f_k)$ is a frame and $\{u_{l k}\}_{l,k\in\mathbb{N}}$ be any bounded linear operator on $l^2(\mathbb{N})$, then the sequence $(\phi_l)$, where $\phi_l=\sum_{k=1}^{\infty} u_{l k} f_k$ forms a frame for $H$ if and only if there exists a constant $C>0$ such that
 \begin{align*}
 \sum_{l=1}^{\infty} |\langle \phi_l, f \rangle|^2\geq C \sum_{k=1}^{\infty}|\langle f_k, f\rangle|^2, \ \text{for all} \ f\in H.
 \end{align*}
  \\ We can prove the following computable version of the above.
  
\begin{theorem}
	Let $H$ be a computable Hilbert space and $(f_i)$ be a computable frame for $H$ with computable analysis operator. Then for any frame $(\phi_l)$ in $H$ with computable analysis operator, the operator $U_{\phi}:l^{2}\rightarrow l^{2}$ given by
	\[\left(\begin{array}{ccccc}
	\langle\phi_{1},S^{-1}f_{1}\rangle & \langle\phi_{1},S^{-1}f_{2}\rangle & \cdots\\
	\langle\phi_{2},S^{-1}f_{1}\rangle & \langle\phi_{2},S^{-1}f_{2}\rangle & \cdots\\
	\vdots &\vdots &\vdots & \end{array}\right)\] is a computable operator on $l^2$.
\end{theorem}
\begin{proof}
 The map $U_\phi$ is a bounded linear operator on $l^2$ by Proposition 5.5.6[7]. The sequence $(S^{-1}f_k)_k$ is a computable sequence in $H$ and since the analysis operator corresponding to the frame $(\phi_l)$ is computable, the sequence $((\langle\phi_l, S^{-1} f_k\rangle)_l)_k$ = $(U_\phi(\delta_k))_k$ is a computable sequence in $l^2$. Altogether, this implies that $U_\phi$ is a computable operator on $l^2$.
\end{proof}
 In the following, we prove a more uniform version of the above result. 
\begin{theorem}
 Let $(f_{k})$ be a computable frame for $H$ with computable analysis operator where $H$ is a computable Hilbert space. Then the map 
 \begin{align}
 G: \subseteq  \ &C(H,l^2) \times \mathbb{R} \rightarrow C(l^{2},l^{2})\\
 &(T^{*},s)\mapsto U_\phi \end{align}
 with dom($G$)= $\{(T^{*},s)|\ T^{*} \ \text{is the analysis operator corresponding to a frame}\  (\phi_i) \ \text{in} \ H \\ \text{and} \ s > 0 \    \text{is such that}\  \lVert U_\phi \rVert \leq s\}$ is computable.
 \end{theorem}
 \begin{proof}
  Given $[\delta_H \rightarrow \delta_{l^2}]$ name of $T^{*}$, $\delta_\mathbb{R}$ name of $s$ where $T^{*}$ is the analysis operator corresponding to a  frame $(\phi_{i})$ in $H$ and $s$ is an upper bound to the norm of the corresponding operator $U_\phi$ and given some $x\in l^{2}$ and precision $m\in \NN$, we can effectively find $n \in \NN$ and numbers $q_0, q_1,...,q_n\in Q_\mathbb{F}$ such that
\begin{align*}
||x-\sum_{i=0}^n q_i \delta_i||<\frac{2^{-m}}{s}.
\end{align*}
It follows that
\begin{align*}
||U_{\phi} x-U_{\phi} (\sum_{i=0}^n q_i \delta_i)||< 2^{-m}.
\end{align*}
Since $U_\phi$ is linear, $U_\phi(\sum_{i=0}^n q_i \delta_i)=\sum_{i=0}^n q_i U_\phi(\delta_i)=\sum_{i=0}^n q_i(\langle \phi_l, S^{-1} f_i \rangle)_l$. Using Evaluation and Type conversion, we obtain the computability of the map $G$.
\end{proof}

\par Conversely, we can have the following result.
\begin{theorem}
	Let $(f_k)$ be a computable frame for a computable Hilbert space $H$. Suppose $U=\{u_{n k}\}_{n,k\in\mathbb{N}}$ be any bounded linear operator on $l^2(\mathbb{N})$ such that $U^*$ is computable. Then the frame $(\phi_n),$ where $\phi_n=\sum_{k=1}^\infty u_{n k} f_k$, forms a computable frame in $H$.
\end{theorem}
\begin{proof}
Since the map $U^*$ is computable, we get the computability of the map $n\mapsto U^*(e_n)=(u_{n k})_k$. As the synthesis operator associated to a computable frame is computable, we obtain the computability of the map $n \mapsto T((u_{n k})_k)=\sum_{k=1}^{\infty} u_{n k} f_k$. Thus, the frame $(\phi_n)$ where  $\phi_n=\sum_{k=1}^{\infty} u_{n k} f_k$ forms a computable frame in $H$.
\end{proof}

 The condition of computability of the operator $U^*$ cannot be relaxed as is justified by the following example.
\begin{example}
Let $(f_k)=(\delta_k)$ be a computable frame in $l^2$ where $(\delta_k)$ is the standard orthonormal basis of $l^2$. Define a bounded linear operator $U$ on $l^2$ by
 \(\left(\begin{array}{ccccc}
    1 & 0 & 0 & 0 & \cdots\\
    0 & 1 & a_1 & a_2 & \cdots\\
    0 & 0 & 1 &0 & \cdots\\
    \vdots &\vdots &\vdots &\vdots & \vdots \end{array}\right)\),
where $(a_i)$ is a computable sequence of positive reals such that $||(a_i)||_{l^2}$ exists but is not computable. We assume that $a_{0}=1$ and $||(a_i)||_{l^2}^{2} < 2$. Then the sequence $(\phi_n)=(U^*(\delta_n))$ forms a frame which is not computable.
\end{example}

The frame decomposition \cite{6}, $f= \sum_{k=1}^{\infty}\langle f,S^{-1}f_{k} \rangle f_{k}$, for all $f\in H$ shows that every element in $H$ has a representation as an infinite linear combination of the frame elements. If $(f_{k})$ is an overcomplete frame, then there exists frames $(g_{k})\neq (S^{-1}f_{k})$ for which
  \begin{align*} f= \sum_{k=1}^{\infty}\langle f,g_{k} \rangle f_{k} \ \ \textrm{for all} \ \  f\in H.\end{align*}
A frame $(g_{k})$ satisfying the above representation is called a \emph{dual frame} of $(f_{k})$. The above condition is equivalent to
 $f= \sum_{k=1}^{\infty} \langle f,f_{k} \rangle g_{k} \ \ \textrm{for all} \ \ f\in H$.
The frame $(S^{-1}f_{k})$ is called the \emph{canonical dual frame} of $(f_{k})$.\\
We now extend this idea in the computability theory.
\begin{definition}
Let $H$ be a computable Hilbert space and $(f_{i})$ be a computable frame for $H$. A computable frame $(g_{i})$ for $H$ is called a \emph{computable dual frame} for $(f_{i})$ if it satisfies
\begin{align} f= \sum_{i=1}^{\infty}\langle f,g_{i} \rangle f_{i}, \ \ \textrm{for all} \ \  f\in H. \end{align}
\end{definition}

For the canonical dual frame, the analysis operator and the synthesis operator are computable by the following result.

\begin{theorem}
Let $H$ be a computable Hilbert space. Let $(f_{n})$ be a computable frame for $H$ with computable analysis operator $T^{*}$. Then, the analysis operator $\widetilde{T}^{*}$ and the synthesis operator $\widetilde{T}$ of the canonical dual frame $(\widetilde{f_{n}})$ = $(S^{-1}f_{n})$ are computable operators.
\end{theorem}
\begin{proof}
 The computability of $S^{-1}$ follows from the computability of $T^{*}$ and the synthesis operator $T$ and so $(\widetilde{f_{n}})$ = $(S^{-1}f_{n})$ is a computable dual frame. By Theorem 3.5, the associated synthesis operator $\widetilde{T}$ is computable. The relation $\widetilde{T}^{*}$ = $T^{*}S^{-1}$ implies the computability of the operator $\widetilde{T}^{*}$.
\end{proof}

We know that given a frame, the associated analysis operator maps $H$ bijectively onto a closed subspace of $l^{2}$. The next result shows that the orthogonal projection onto this closed subspace is a computable operator.

\begin{theorem}
Let $H$ be a computable Hilbert space. Let $(f_{n})$ be a computable frame for $H$ with computable analysis operator $T^{*}$. Then, the orthogonal projection $P$ of $l^{2}$ onto the range of $T^{*}$, $T^{*}(H)$, is a computable operator.
\end{theorem}
\begin{proof}
 The orthogonal projection $P$ of $l^{2}$ onto $T^{*}(H)$ is given by
 \begin{align*}  P((c_{n}))= (\langle \sum_{n=1}^{\infty}c_{n}\widetilde{f_{n}}, f_{k}\rangle )_{k\in \NN},\ \  (c_{n})\in l^{2}.\end{align*}
 This can be verified by showing that $P$ is identity on $T^{*}(H)$ and zero on     $(T^{*}(H))^{\perp}$ = $Ker(T)$. Now, since $(\widetilde{f_{i}})$ = $(S^{-1}f_{i})$ is a computable dual frame, given $\delta_{l^{2}}$ name of $(c_{n})$, we can compute $\delta_{H}$ name of  $\sum_{n=1}^{\infty}c_{n}\widetilde{f_{n}}$. The computability of $T^{*}$ implies that $\delta_{l^{2}}$ name of $(\langle \sum_{n=1}^{\infty}c_{n}\widetilde{f_{n}}, f_{k}\rangle)_{k\in \NN}$ can be computed. As $T^{*}(H)$ is a computable subspace of $l^{2}$, we can get $\delta_{T^{*}(H)}$ name of $P((c_{n}))$.
 \end{proof}

Next, we prove a computable version of Theorem 3.3.2 \cite{6} related to the duality of basis in $H$.
\begin{theorem}
Let $H$ be a computable Hilbert space and $(e_{k})$ be a computable basis for $H$ such that the sequence of norms of the coordinate functionals $(\|e_{k}^{'}\|)$ is computable. Then, there exists a unique computable basis $(g_{k})$ for $H$ such that
\begin{align*} f= \sum_{k=1}^{\infty}\langle f,g_{k}\rangle e_{k}. \end{align*}
\end{theorem}
\begin{proof}
 Since $(e_{k})$ is a computable basis for $H$, by Proposition 3.3 \cite{5}, we obtain the computability of the sequence of coordinate functionals $(e_{k}^{'})$ with respect to $[\delta_{H}\rightarrow \delta_{\mathbb{F}}]$ representation, where $e_{k}^{'}: H \rightarrow \mathbb{F}$ is given by
 \begin{align*}  e_{k}^{'}(\sum_{i=1}^{\infty}x_{i}e_{i})= x_{k},\ \ k\in \NN. \end{align*}
Since $(\|e_{k}^{'}\|)$ is computable, we get that $(e_{k}^{'})$ is a computable sequence with respect to $\delta_{H^{'}}$ representation. By the computable Fr{\'e}chet Riesz Theorem \cite{3}, there exists a computable sequence $(g_{k})$ in $H$ such that $e_{k}^{'}(f)= \langle f, g_{k} \rangle$ for all $f\in H$. 
Thus, we obtain
\begin{align*} f= \sum_{k=1}^{\infty}\langle f,g_{k}\rangle e_{k}, \ \ \textrm{for all} \ \ f\in H, \end{align*}
where $(g_{k})$ is a unique computable basis for $H$.\end{proof}

\begin{corollary}
Let $H$ be a computable Hilbert space and $(e_{k})$ be a computable monotone basis for $H$. Then, there exists a unique computable basis $(g_{k})$ for $H$ such that
 $f= \sum_{k=1}^{\infty}\langle f,g_{k} \rangle e_{k}, \ \  f\in H$.
\end{corollary}
\begin{proof}
 The proof follows in view of Proposition 4.6 \cite{5}.
\end{proof}

\begin{corollary}
Let $H$ be a computable Hilbert space and $(e_{k})$ be a computable orthonormal basis for $H$. Then, there exists a unique computable basis $(g_{k})$ for $H$ such that
$f= \sum_{k=1}^{\infty}\langle f,g_{k}\rangle e_{k},\ \  f\in H$.\\
\end{corollary}

\par Now, if $(f_{k})= (Ue_{k})$ is a computable frame for a computable Hilbert space $H$, with $U^{*}$ computable, then Corollary 3.10 shows that $S^{-1}$ is a computable operator and so $(S^{-1}f_{k})$ is a computable frame such that
$f= \sum_{k=1}^{\infty}\langle f,S^{-1}f_{k}\rangle f_{k}$ for all $f\in H$.
That is, $(S^{-1}f_{k})$ is a computable dual frame for $(f_{k})$. But such a dual frame is not unique as shown in Lemma 5.6.1 \cite{6}.\\

In the following example, we show that the dual of a computable frame need not be computable.

\begin{example}
Let $(a_{i}) \in l^{2}$ be a sequence of positive real numbers that is computable as a sequence in $\RR$ such that $\|(a_{i})\|_{l^{2}}$ is not computable. We assume $a_{0}= 1$, $\|(a_{i})\|_{l^{2}}^{2} < 2$ and equip $l^{2}$ over $\RR$ with its standard basis $(\delta_{i})$. Define $(g_{i})\in l^{2}$ by
\begin{align*} g_{i}= (0,0,....0,1,a_{1},a_{2}....), \end{align*}
for all $i\in \NN$. That is, $g_{i}= U\delta_{i}$, where $U$ is given by

\[\left(\begin{array}{ccccc}
 1& 0 & 0 & 0 & \cdots\\
 a_{1}& 1 & 0 & 0 & \cdots\\
 a_{2} & a_{1} & 1 & 0 & \cdots\\
 \vdots &\vdots &\vdots & \ddots\end{array}\right).\]

Then, $(g_{i})$ is a non computable frame. Also $x= (x_{0}, x_{1},...)\in l^{2}$ can be expressed as
\begin{align*} x &= \sum_{i=0}^{\infty}(x_{i}-\sum_{j=0}^{i-1}a_{i-j}x_{j})g_{i} \\
                  &= \sum_{i=0}^{\infty}\langle x,(-a_{i},-a_{i-1},...-a_{1},1,0...)\rangle g_{i}.\end{align*}
Define $f_{i}= (-a_{i},-a_{i-1},...-a_{1},1,0...)$, $i\in \NN $. Then, $(f_{i})\in l^{2}$ is a computable dual frame for $(g_{i})$. \end{example}

\begin{remark}
The condition of computability of the sequence $(\|e_{k}^{'}\|)$ in Theorem 3.24, cannot be relaxed. Indeed, in Example 3.27, the sequence $(f_{i})$ is a computable basis for $l^{2}$, $(g_{i})$ is a non computable basis for $l^{2}$  and the coordinate functional $f_{0}^{'}$ has norm $\|f_{0}^{'}\|$= $\|(a_{i})\|_{l^{2}}$ which is not computable.
\end{remark}

Finally, we give a computable version of a necessary and sufficient condition for the dual frame to be computable.

\begin{theorem}
Let $H$ be a computable Hilbert space and $(f_{k})$ be a computable frame for $H$. Then, a sequence $(g_{k})$ in $H$ is a computable dual frame of $(f_{k})$ if and only if $(g_{k})= (V\delta_{k})$, where $V: l^{2}\rightarrow H$ is a computable left inverse of $T^{*}$ and $(\delta_{k})$ is the standard computable orthonormal basis of $l^{2}(\NN)$.
\end{theorem}
\begin{proof}
Let $(g_{k})\subseteq H$ be a computable dual frame of $(f_{k})$. Then $f= \sum_{k=1}^{\infty}\langle f,f_{k}\rangle g_{k}$ for all $f\in H$.
Let $V$ be the synthesis operator of $(g_{k})$. That is, $V: l^{2}\rightarrow H$ is given by
$V((c_{k}))= \sum_{k=1}^{\infty}c_{k}g_{k}$, for all $(c_{k})\in l^{2}$.
Then by Theorem 3.5, $V$ is computable. Also, the operator $V$ satisfies
$V(\delta_{k})= g_{k}$ for all  $k\in \NN$
 and
 $VT^{*}(f)= \sum_{k=1}^{\infty}\langle f,f_{k} \rangle g_{k}= f$ for all $f\in H$.
Hence $V$ is a computable left inverse of $T^{*}$.\\
Conversely, let $(g_{k})= (V\delta_{k})$, $k\in \NN$, where $V$ is a computable left inverse of $T^{*}$. Since $V(T^{*}f)=f$ for all $f\in H$, $V$ is surjective and so $(g_{k})$ is a frame for $H$. Also, we have
$f=  \sum_{k=1}^{\infty}\langle f,f_{k} \rangle g_{k}$ for all $f\in H$.
This shows that $(g_{k})$ is a dual frame to $(f_{k})$. Finally, the computability of $(g_{k})$ follows from the computability of $V$ and of the sequence $(\delta_{k})$.\end{proof}

\par The following is a result on characterization of computable dual frames. Here, a computable Bessel sequence is a Bessel sequence that is computable as a sequence in $H$.
\begin{theorem}
Let $(f_k)$ be a computable frame for a computable Hilbert space $H$, with computable analysis operator. The computable dual frames of $\{f_k\}$ are precisely the families
\begin{align*}
\{g_k\}=\bigg\{S^{-1} f_k+h_k-\sum_{j=1}^{\infty} \langle S^{-1} f_k, f_j \rangle h_j \bigg\}_{k=1}^{\infty},
\end{align*}
where $\{h_k\}$ is a computable Bessel sequence in $H$.
\end{theorem}
\begin{proof}
By Theorem 3.29, the computable dual frames of $(f_k)$ are precisely the families $(V \delta_k)$, where $V:l^2\rightarrow H$ is a computable left inverse of $T^*$. It can be easily observed that the computable left inverses of $T^*$ are operators of the form $S^{-1}T+W(I-T^*S^{-1}T)$, where $W:l^2\rightarrow H$ is a computable operator, that is, $W:l^2\rightarrow H, W((c_j))=\sum_{j=1}^{\infty}c_j h_j,$ where$\{h_k\}$ is a computable Bessel sequence in $H$. Using the fact that $T(\delta_k)=f_k$, we obtain the desired result.
\end{proof}

\mbox{}
\end{document}